\documentclass[1p]{elsarticle}
\usepackage{tabularx,array,amsmath}
\usepackage{amsfonts, amsthm}
\usepackage{amssymb}
\newtheorem{theorem}{Theorem}[section]
\newtheorem{lemma}[theorem]{Lemma}
\newtheorem{proposition}[theorem]{Proposition}
\newtheorem{corollary}[theorem]{Corollary}

\theoremstyle{definition}
\newtheorem{definition}[theorem]{Definition}
\journal{Journal of Algebra}

\DeclareMathOperator{\GL}{GL}
\begin{document}
\newcommand{\TT}{\mathbf{\Theta}}
\newcommand{\dd}{\sqrt{d}}
\newcommand{\modd}{\mbox{mod} \ }
\newcommand{\Hz}{H(O_d)}

\title{Counting Irreducible Representations \\ of the  Heisenberg Group Over the Integers of a Quadratic Number Field}
\author{Shannon Ezzat}


\begin{abstract}
We calculate the representation zeta function of the Heisenberg group over the integers of a quadratic number field. In general,
the representation zeta function of a finitely generated torsion-free nilpotent group enumerates equivalence classes of representations, called twist-isoclasses. This calculation is based on an explicit description of a representative from each twist-isoclass. Our method of construction involves studying the eigenspace structure of the elements of the image of the representation and then picking a suitable basis for the underlying vector space.
\end{abstract}
\maketitle

\section{Introduction}
Let $G$ be a finitely generated torsion-free nilpotent group. Let $\chi$ be a 1-dimensional complex representation and $\rho$ an $n$-dimensional complex representation of $G$. We define the product $\chi  \otimes \rho$ to be a \emph{twist} of $\rho$. Two representations $\rho$ and $\rho_*$ are \emph{twist-equivalent} if, for some 1-dimensional representation $\chi$, $\chi \otimes \rho \cong \rho_*.$ This twist-equivalence is an equivalence relation on the set of irreducible representations of $G$. In \cite{LM} Lubotzky and Magid call the equivalence classes \emph{twist-isoclasses}. They also show that there are only finitely many irreducible $n$-dimensional complex representations up to twisting and that for each $n \in \mathbb{N}$ there is a finite quotient $N$ of $G$ such that each $n$-dimensional irreducible representation $\rho$ of $G$ is twist-equivalent to one that factors through $N$. Henceforth we call the complex representations of $G$ simply representations. We denote the number of twist-isoclasses of irreducible representations of dimension $n$ by $r_n(G)$ or $r_n$ if no confusion will arise.

Consider the formal Dirichlet series \begin{equation} \zeta_G(s) = \sum_{n=1}^\infty r_n(G) n^{-s}.
 \end{equation}
If $G$ is a finitely generated torsion-free nilpotent group then by \cite[Lemma 2.1]{SV}, $\zeta_G(s)$ converges
 on a right half plane of $\mathbb{C}$, say $D$, where $D := \{s \in \mathbb{C} \mid \Re(s) > \alpha\}$ for some $\alpha \in \mathbb{R}$; we call $\zeta_G: D \to \mathbb{C}$ the \emph{representation zeta function} of $G$. Note that in \cite{Voll1}, \cite{Voll2}, and \cite{SV} this function is denoted $\zeta_G^{irr} (s)$. For a prime $p,$ define \begin{equation} \zeta_{G,p} (s) := \sum_{n=0}^\infty r_{p^n}(G) p^{-ns}, \end{equation} to be the \emph{$p$-local representation zeta function} of $\zeta_G(s).$ Since $G$ is nilpotent, each finite quotient decomposes as a direct product of its Sylow-$p$ subgroups. Since the irreducible representations of a direct product of finite groups
 are tensor products of irreducible representations of its factors, we have the Euler product $\zeta_G (s) = \prod_{p} \zeta_{G,p} (s)$ \cite[Introduction]{SV}.  Moreover, it was shown by Hrushovski and Martin \cite[Theorem 8.4]{HM} that each $p$-local representation zeta
 function is a rational function in $p^{-s}.$

The idea of using zeta functions to study the representation growth of groups is motivated by subgroup growth, where one uses zeta functions to count finite index subgroups. The study of subgroup growth of finitely generated nilpotent groups by zeta functions was introduced in \cite{GSS}. In that paper, the authors calculate the normal subgroup zeta function of the Heisenberg group over the ring of integers of a number field of degree at most two \cite[Prop. 8.2]{GSS}. These zeta functions were given, in part, in terms of the Dedekind zeta function of the associated number field. Research on subgroup zeta functions of nilpotent groups continued in papers such as \cite{duS1, duS2, duSG}.

One can study the growth rate of the sequence $r_n(G)$ without necessarily explicitly constructing the zeta function. We call this study \emph{representation growth.} Representation growth has been used to study other classes of groups. We briefly mention some work done in these areas. For the following groups, $r_n(G)$ counts irreducible representations, as opposed to twist-isoclasses as with nilpotent groups. The idea of using zeta functions to study representation growth was introduced in \cite{Witten}, in which Witten studies compact Lie groups.
Later, representation growth was studied for $S$-arithmetic groups by Lubotzky and Martin \cite{LMar} and, using the language of representation zeta functions, Larsen and Lubotzky \cite{LL}.
Jaikin, in \cite{JZ},
studies representation zeta functions of compact $p$-adic analytic groups with property FAb. In \cite{Voll4} (and the research announcement \cite{Voll3}) Avni et al.\ study compact $p$-adic analytic groups and arithmetic groups using representation zeta functions, 
proving \cite[Corollary D]{Voll4} a conjecture of Larsen and Lubotzky \cite[Conjecture 1.5]{LL}.  In \cite{Avni} Avni shows that arithmetic groups in characteristic zero which satisfy the congruence subgroup property have representation zeta functions with rational abscissa of convergence. Bartholdi and de la Harpe study representation zeta functions of wreath products with finite groups \cite{BdlH}. In \cite{KN}, Kassabov and Nikolov study representation growth of some profinite groups. Craven gives lower bounds for representation growth for profinite and pro-$p$ groups \cite{Craven}.

Representation growth of finitely generated nilpotent groups has been studied in \cite{Voll1} by Voll and \cite{SV} by Stasinski and Voll. Very few examples of representation zeta functions of finitely generated nilpotent groups have appeared in the literature (see \cite[Theorem B]{SV} and \cite[Example 8.12]{HM}). Let
\begin{equation}\label{Heisgroup}
H(\mathbb{Z}):= \langle x,y,z \mid [x,y]=z \rangle
 \end{equation} be the Heisenberg group over the rational integers and let $\zeta(s)$ be the Riemann zeta
 function.
Then, by \cite[Theorem 8.12]{HM},
$\zeta_{H(\mathbb{Z})}(s)$ is given by
\begin{equation}\label{Heis}
\zeta_{H(\mathbb{Z})}(s) = \frac{\zeta(s-1)}{\zeta(s)}
\end{equation}
(the coefficients $r_{p^n}(H({\mathbb Z}))$ were originally calculated in \cite[Theorem 5]{NM}). This has a simpler shape than the corresponding normal subgroup zeta function in \cite[Chapter 15]{LS}.

Let $d$ be a square-free integer and define $O_d$ to be the ring of integers of the number field $\mathbb{Q}(\sqrt{d})$. The Heisenberg group over $O_d$, which we denote $\Hz$, is the group of $ 3 \times 3$ upper unitriangular matrices with entries in $O_d$.

Let $\zeta^{\mathbf{D}}_Q$ be the Dedekind zeta function of a number field $Q.$ The main theorem of this paper is as follows:

\begin{theorem}\label{main}
The representation zeta function of $\Hz$ is
$$\zeta_{\Hz} (s) = \dfrac {\zeta^{\mathbf{D}}_{\mathbb{Q}(\dd)} (s-1)}{\zeta^{\mathbf{D}}_{\mathbb{Q}(\dd)} (s)}.$$
\end{theorem}
We remark that  an analogous theorem holds for the Heisenberg group over the rational integers; that is, if $\mathbb{Q}(\sqrt{d})$ is replaced by $\mathbb{Q}$ (see Equation \ref{Heis}).

We can use this theorem to determine the abscissa of convergence of these representation zeta functions.

\begin{corollary}
For any quadratic number field $\mathbb{Q}(\dd)$ the abscissa of convergence of $\zeta_{\Hz}(s)$ is $2.$
\end{corollary}

Let $\mathbf{G}$ be a group scheme defined over
the ring of integers $O$ of a number field $K$.  For $O'$ a ring extension of $O$, denote the group of $O'$-points of $\mathbf{G}$ by $S(O')$.  Groups of the form $S(O')$ give an important class of examples for representation growth.  The papers
\cite{LL, avnietal, Voll4, Avni}
deal with semisimple $\mathbf{G}$, while \cite{SV} deals with unipotent $\mathbf{G}$.

Suppose the group scheme $\mathbf{G}$ is unipotent. Stasinski and Voll, in \cite[Theorem A]{SV} and \cite[Remark 2.3]{SV}, show that for non-zero prime ideals $P \lhd O^\prime$ the following Euler factorization holds:
\begin{equation}
\zeta_{\mathbf{G}(O^\prime)}(s)=\prod_P \zeta_{\mathbf{G}(O^\prime),P}(s)
\end{equation}
where $\zeta_{\mathbf{G}(O^\prime),P}(s)$ counts continuous representations of $\mathbf{G}(O^\prime_P)$
and $O^\prime_P$ is the completion of $O^\prime$ at $P.$
For group schemes associated to a nilpotent Lie lattice the authors also show that, for almost all prime ideals, the local representation zeta functions behave uniformly under extension of scalars; that is, if $O^\prime_P$ is a finite extension of $O_P$ then $\zeta_{\mathbf{G}(O^\prime)}(s)$ comes from a rational function whose form depends only on $\mathbf{G}$ and not on $O^\prime_P$ (see \cite{SV} for a precise statement).
Corollary 1.3 of \cite{SV} also tells us that the $P$-local representation zeta functions $\zeta_{\mathbf{G}(O^\prime),P}(s)$ satisfy a functional equation which is a refinement of
\cite[Theorem D]{Voll1}
and, additionally, that these local zeta functions are rational functions in $q^{-fs},$ where $q$ is the cardinality of the associated residue field and $f$ is the relative degree of inertia. Moreover, \cite[Theorem B]{SV} calculates $P$-local representation zeta functions of three families of groups arising from unipotent group schemes; indeed, the representation zeta function of $\Hz$, Theorem \ref{main} in our paper,  appears as a special case of their result
(note that we can identify $H(O)$ with the group of $O$-points of a unipotent group scheme $H$ defined over ${\mathbb Z}$).

As a special case of \cite[Theorem B]{SV}, Stasinski and Voll prove the following result, which generalizes Theorem \ref{main}.

\begin{theorem}\label{conjecture}
Let $K$ be an arbitrary algebraic number field and $O$ its ring of integers. Then

$$ \zeta_{H(O)}(s) = \dfrac{\zeta^{\mathbf{D}}_K (s-1)}{\zeta^{\mathbf{D}}_K (s)}. $$
\end{theorem}
We remark that Theorem \ref{conjecture} was stated as a conjecture in an earlier version of the present paper, before the appearance of \cite{SV}.

When the group scheme $\mathbf{G}$ is semisimple there are some similar results to the unipotent case discussed above.
As in the unipotent case, Larsen and Lubotzky \cite[Proposition 1.3]{LL} showed that, for semisimple arithmetic groups satisfying the congruence subgroup property, there is a decomposition of representation zeta functions into local factors
(with an extra factor coming from the prime at infinity).
Work by Avni et al.,
\cite[Theorem A]{Voll3} \cite[Theorem A]{Voll4}, has proven that the representation zeta functions of a certain class of compact $p$-adic analytic groups related to $\mathbf{G}$ 
also behave uniformly under extension of scalars.



From now on we restrict our attention to finitely generated torsion-free nilpotent groups. In the paper \cite{Voll1}, Voll introduces a method to parameterize representations in order to calculate $p$-local representation zeta functions. One counts twist-isoclasses of a finitely generated torsion-free nilpotent group $G$ by counting the coadjoint orbits of the Lie ring associated to $G$. However, this Kirillov orbit method has the restriction that
in general it excludes a finite number of primes. We note that both the Kirillov orbit method and the method in
the present paper reduce to counting solutions to polynomial equations, although in the Kirillov orbit method this counting is disguised in the machinery of Igusa local zeta functions.

The method used in
the present paper is quite general.
Our techniques could, in principle, be used to calculate all $p$-local representation zeta functions, and therefore the global representation zeta function, of any finitely generated torsion-free nilpotent group. However, in practice this may not be
feasible---as the complexity of the eigenspace structure of the irreducible representations increases, the complexity of the calculation would increase quite quickly.  Note that our method relies on less mathematical machinery than the Kirillov orbit method and thus can be appreciated with minimal technical background;
moreover, it explicitly gives all of the irreducible representations. The deeper methods from \cite{Voll1} and \cite{SV} allow for easier computations in many cases since, in these methods, one counts characters without constructing the representations explicitly.


The methods of \cite{Voll1} and \cite{SV} are not valid for all primes for an arbitrary nilpotent group. It is worth noting, however, that \cite[Section 2.4]{SV} introduces a method for class-2 finitely generated nilpotent groups obtained from unipotent group schemes, including the groups studied in 
the present paper, that is valid for all primes. The constructive method introduced
here differs from the previously mentioned methods.
The author will use it to calculate the
full global representation zeta functions of the  groups $M_n:=\langle y,x_1, \ldots, x_n \mid [y,x_i]=x_{i+1}, i= 1, \ldots, n-1\rangle, n \in \{3,4\}$ of nilpotency
class $n$ in a forthcoming paper.

In \cite{Voll1} it was discovered that representation zeta functions satisfy the following functional equation:

\begin{theorem} {\cite[Theorem D]{Voll1}}\label{voll}
Let $G$ be a finitely generated torsion-free nilpotent group with derived subgroup $G'$ with Hirsch length $d'$. Then, for almost all primes,
\begin{equation}
\zeta_{G,p} (s)|_{p\rightarrow p^{-1}} = p^{d'} \zeta_{G,p} (s).
\end{equation}
\end{theorem}
\noindent This functional equation has been refined, as previously mentioned, by Stasinski and Voll in \cite[Theorem A]{SV}.  They give a functional equation for $\zeta_{\Hz,p}(s)$ in \cite[Corollary 1.3]{SV}.  We note that $\zeta_{\Hz,p} (s)$ does indeed satisfy the functional equation of Theorem \ref{voll} if $p$ is not ramified.

Let $p$ be a rational prime and write $p=\prod_{i=1}^j (P_i)^{e_i},$ where $j \in \{1,2\},$  $P_i $ is a prime ideal of $O_d,$ and $e_i$ is the ramification index of $P_i$.  We call $\prod_{i=1}^j \zeta^\mathbf{D}_{\mathbb{Q}(\sqrt{d}),P_i}$ the $p$-local Dedekind zeta function.
It is easy to see that the Euler product of the $p$-local representation zeta functions of $\Hz$ can be refined to an Euler product of $P_i$-local Dedekind zeta functions by factorization; that is,
 \begin{equation}
 \zeta_{\Hz,p}(s)=\prod_{i=1}^j \dfrac{\zeta^\mathbf{D}_{\mathbb{Q}(\sqrt{d}),P_i}(s-1)}{\zeta^\mathbf{D}_{\mathbb{Q}(\sqrt{d}),P_i}(s)}.
 \end{equation}
The functional equation in \cite[Corollary 1.3]{SV} can be obtained by a short computation of the functional equations satisfied by the $P$-local Dedekind zeta functions.

The structure of the paper is as follows. Section 2
introduces some definitions and notation. Section 3 contains the proof of the main theorem which is separated into three parts: studying eigenspace behaviour, picking a suitable basis, and counting the number of twist-isoclasses, respectively.
\section{Preliminaries}

We remind the reader of some standard concepts in algebraic number theory. For more details see, for example, \cite[Section 10.5]{Cohen}.
\begin{definition} Let $K$ be an algebraic number field and $O_K$ its ring of integers. Then
$\zeta^{\mathbf{D}}_K (s) = \sum_{I \subseteq O_K} (N_{K/\mathbb{Q}} (I))^{-s}$ is the \emph{Dedekind zeta function} of $K$ where $I$ runs through the non-zero ideals of $O_K$ and $N_{K/\mathbb{Q}} (I)$ is the norm of $I$ with respect to $\mathbb{Q}$.
\end{definition}

The zeta function $\zeta^{\mathbf{D}}_K (s)$ has an Euler product decomposition

$$\zeta^{\mathbf{D}}_K (s) = \prod_{P \subseteq O_K} \frac{1}{1 - (N_{K/\mathbb{Q}}(P))^{-s}},$$
\noindent where $P$ runs over all non-zero prime ideals of $O_K$. This decomposition reflects the unique factorization of ideals of $O_K.$

We recall another standard definition \cite[Section 3.3]{Cohen}.

\begin{definition}
Let $p$ be a rational prime, and consider the ideal $(p) \lhd O_d$. If $(p)$ is prime then $p$ is \emph{inert.} If $(p)$ is the product of two distinct prime ideals then $p$ \emph{splits.} If $(p)$ is the square of a prime ideal then $p$ is \emph{ramified}. \end{definition}

This definition is equivalent to the equation $x^2- \Delta \equiv 0 \ (\modd p)$ having $0,2,$ or $1$ solution, respectively and where $\Delta$ is the discriminant of $\mathbb{Q}(\sqrt{d}).$ We note that $\Delta=4d$ if $d \equiv 2,3 \ (\modd 4)$ and $\Delta = d$ if $d \equiv 1 \ (\modd 4).$ We also note that there are only a finite number of ramified primes and a prime $p$ is ramified if and only if it divides the discriminant of the number field.

The following is a well known result.

\begin{proposition} \label{oror}
Let $p$ be a prime and $\mathbb{Q}(\dd)$ be a quadratic number field. Then the $p$-local Dedekind zeta function of $\mathbb{Q}(\dd)$ is

\begin{equation}
\zeta^{\mathbf{D}}_{\mathbb{Q}(\dd),p}(s) =
\begin{cases}

\dfrac{1}{1-p^{-2s}} &\text{if $p$ is inert},\\
\\
 \left(\dfrac{1}{1-p^{-s}}\right)^2 &\text{if  $p$ splits},\\

\\
\dfrac{1}{1-p^{-s}} &\text{if $p$ is ramified}. \\
\\

\end{cases}
\end{equation}

\end{proposition}

To aid our study of the representation zeta function of $\Hz$ we choose an appropriate presentation. It is easily seen that the following six matrices generate $\Hz$:

\begin{align*}
x= \left( \begin{array}{ccc}
1 & 1 & 0 \\
0 & 1 & 0 \\
0 & 0 & 1
\end{array} \right)  \qquad & x_d= \left( \begin{array}{ccc}
1 & D & 0 \\
0 & 1 & 0 \\
0 & 0 & 1
\end{array}\right) \\
\\
y= \left( \begin{array}{ccc}
1 & 0 & 0 \\
0 & 1 & 1 \\
0 & 0 & 1
\end{array} \right)  \qquad & y_d= \left( \begin{array}{ccc}
1 & 0 & 0 \\
0 & 1 & D \\
0 & 0 & 1
\end{array}\right) \\
\\
z= \left( \begin{array}{ccc}
1 & 0 & 1 \\
0 & 1 & 0 \\
0 & 0 & 1
\end{array} \right)  \qquad & z_d= \left( \begin{array}{ccc}
1 & 0 & D \\
0 & 1 & 0 \\
0 & 0 & 1
\end{array}\right) \\
\end{align*}
where $D = \sqrt{d}$ if $d \equiv 2,3 \ (\modd 4)$ and $D = \frac{1+\sqrt{d}}{2}$ if $d \equiv 1 \ (\modd 4)$. Note that $(1,D)$ is a $\mathbb{Z}$-basis for $O_d.$\\

 It can be shown that a presentation for this group is given by $$\langle x,x_d,y,y_d,z,z_d \mid [x,y]=z, [x,y_d] =[x_d,y] =z_d, [x_d,y_d]=z^{d} \rangle$$  if  $d \equiv 2,3 \ (\modd 4)$ and $$\langle x,x_d,y,y_d,z,z_d \mid [x,y]=z, [x,y_d] =[x_d,y] =z_d,[x_d,y_d] =z^{\frac{d-1}{4}}z_d \rangle$$ if $d \equiv 1 \ (\modd 4)$. By convention, commutators that cannot be deduced from the relations that appear are trivial.\\



Before we prove the main result of the paper, we tabulate the notation used herein for easy reference:\\

\begin{tabular}{l|l}
$\Hz$ & the Heisenberg group over the integers of $\mathbb{Q}(\dd)$\\
$d$ & {a square-free integer}\\
$D_l$ & $d$ if $d \equiv 2,3 \ (\modd 4)$; $l+\frac{d-1}{4}$ if $d \equiv 1 \ (\modd 4)$\\
$\langle a_1, \ldots, a_k \rangle$ & the group generated by $a_1, \ldots, a_k$ \\
$\phi$ & {the Euler phi function}\\
$(M)_{i,j}$ & the $(i,j)$th entry of matrix $M$\\
$G \cdot x$ & the orbit of $x$ under the action of a group $G$\\
$[a,b]$ & the commutator of group elements $a$ and $b$, that is, $aba^{-1}b^{-1}$\\

\end{tabular}\\

In this paper, matrix entries that are blank are assumed to be $0.$
\section{Proof of Theorem \ref{main}}
%
%
%
%
%
%
%
%
%
%
%

\subsection{Studying Eigenspaces}

We begin the proof of the main theorem by studying the twists of some irreducible representation $\rho : \Hz \to \GL_{p^n}(\mathbb{C}).$

\begin{definition}
A twist-isoclass is of \emph{dimension} $n \in \mathbb{N}$ if the representations in the twist-isoclass are of dimension $n$.
\end{definition}

We note that this is not to be confused with the dimension of a twist-isoclass as a subvariety (see \cite{LM} for details).

It is easy to show that there is only 1 twist-isoclass of dimension 1:
that is, $r_1 = 1$, where $r_1$ is the first coefficient in $\zeta_{\Hz}(s)$.

\begin{lemma} \label{L2}
Let $\rho: \Hz \to \GL_n(\mathbb{C})$ be an irreducible representation and let $J =\{x,y,x_d,y_d\}$. Then there exists a representation $\chi: \Hz \to \GL_1(\mathbb{C})$ such that for each $ j \in J$ we have that $1$ is an eigenvalue of $\chi \otimes \rho(j).$
\end{lemma}

\begin{proof}

Let $\rho: \Hz \to \GL_n(\mathbb{C})$ be an irreducible representation and for each $j \in J$ let $\lambda_j$ be an eigenvalue of $\rho(j)$. We can twist any irreducible representation by any 1-dimensional representation and remain in the same twist-isoclass. We deduce that we can choose a 1-dimensional representation $\chi$ such that $\chi(j)=(\lambda_j)^{-1}$.
\end{proof}

We call a representation \emph{good} if $1$ is an eigenvalue of all of the non-central images of the generators.

We will show that the images of the generators of any irreducible representation of $\Hz$, up to twisting, can be written as matrices in a certain canonical form and that any set of matrices satisfying this form is, in fact, isomorphic to the images of the generators for some irreducible representation of $\Hz$. Finally, if two irreducible representations are not twist-equivalent, their associated canonical forms differ.

Let $p$ be a prime, $n \geq 1$, and $\rho: \Hz \to  \GL_{p^n}(\mathbb{C})$ be a good irreducible representation. Let

\begin{align*}
 A &:= \rho(x) &A_d &:= \rho(x_d) \\
 B &:= \rho(y)& B_d &:= \rho(y_d) \\
 \Lambda &:= \rho(z)&  \Lambda_d &:= \rho(z_d).
\end{align*}

\noindent Our aim in the first two sections is to choose a basis for $\mathbb{C}^{p^n}$ such that the images of our generators in $\GL_{p^n}(\mathbb{C})$ are in a ``nice" form. This basis will be chosen so that $A$ and $A_d$ are diagonal matrices, $B$ and $B_d$ are block permutation matrices, and $\Lambda$ and $\Lambda_d$ are scalar matrices. To avoid extra notation, for the rest of the paper we do not distinguish between a linear operator and its matrix with respect to some basis. Also, we identify scalars with scalar matrices; in particular,  for $\lambda$ a root of unity, we will call the matrix $\lambda I$ a root of unity as well.

\par
Since $z$ and $z_d$ are central in $\Hz$, by Schur's lemma we must have that $\Lambda$ and $\Lambda_d$ are homotheties. By \cite[Theorem 6.6]{LM}, $\rho$, up to twisting, factors through a finite quotient of $\Hz$, say $\overline{\Hz}$. Therefore, without loss of generality, the representation $\rho$ is such that the images of elements of $\Hz$ under $\rho$ must have finite order. Hence, for every $g \in \overline{\Hz}$ we have $ g^k=e$ for some minimal $k$. It then follows that $\rho(g)^k=I$ and the minimum polynomial of $\rho(g)$ is $x^k-1$. Since the $k$th roots of unity are distinct, this polynomial factors over the complex numbers into $k$ distinct linear factors. Thus $A,A_d,B,$ and $B_d$ must be diagonalizable and have eigenvalues which are roots of unity, and $\Lambda$ and $\Lambda_d$ must be roots of unity. It is important to note that twisting does not affect the diagonalizability of the images of the generators; twisting is simply multiplication by scalars. Also, since $[A,A_d] = I$, $A$ and $A_d$ are simultaneously diagonalizable.

\begin{definition}
Let $X$, $Y \in \GL_{p^n}(\mathbb{C})$. If $[X,Y]=Z$ for some homothety $Z \in \GL_{p^n}(\mathbb{C})$ then we say $X$ is \emph{$Z$-arrangeable}, or simply \emph{arrangeable}, under $Y$.
\end{definition}

Denote by $E_{X,\lambda}$ the eigenspace of the linear operator $X$ with eigenvalue $\lambda$. If there will be no confusion we may omit $X.$

\begin{lemma} \label{L3}  Let $X$ be $Z$-arrangeable under $Y$. Then $YE_{X,\lambda}=E_{X,Z\lambda}$.
\end{lemma}

\begin{proof}
 Let $\mathbf{v}$ be an eigenvector of $X$ such that $X\mathbf{v} = \lambda\mathbf{v}$. Then
$$XY\mathbf{v} = ZYX\mathbf{v}= Z Y \lambda \mathbf{v} = Z \lambda Y \mathbf{v}.$$

Therefore $Y\mathbf{v}$ is also an eigenvector of $X$ and $YE_{X,\lambda}=E_{X,Z\lambda}.$ The converse is similar.

\end{proof}

\begin{lemma} \label{eigenspacesize}
 There are $p^r$ eigenspaces of $A$ each of dimension $p^m$ for some $r,m$ where $r+m=n$.
\end{lemma}

\begin{proof}
 Since $[A,B] = \Lambda$ and $[A,B_d]=\Lambda_d$ we have that $A$ is arrangeable under $B$ and $B_d$. By Lemma \ref{L3}, $B$ and $B_d$ must send an eigenspace of $A$ of some dimension, say $\gamma$, to another eigenspace of dimension $\gamma$. Therefore the direct sum of all eigenspaces of dimension $\gamma$ forms a stable subspace of $\rho$. Since $\rho$ is irreducible by assumption, all eigenspaces of $A$ must be of dimension $\gamma$. Lemma \ref{L3}, along with the assumption that $\rho$ is good, lets us additionally conclude that the eigenvalues of $A$ are powers of $\Lambda$.

 Since $A$ is diagonalizable,  $A$ has ${\omega}$ distinct eigenvalues of multiplicity ${\gamma}$ for some $\omega$. Since $p^n = \omega \gamma$ this shows us that $\omega$ and $\gamma$ must be powers of $p$, say $\omega = p^r$ for some $r$ and $\gamma = p^m$ for some $m$.
\end{proof}

We state a simplified version of Lemma \ref{eigenspacesize} without proof.

\begin{lemma} \label{simplified}
Let $X,Y,Z \in \GL_{p^r}(\mathbb{C})$, where $Z$ is a primitive $p^r$th root of unity, such that $[X,Y]=Z$. Then there are $p^r$ distinct 1-dimensional eigenspaces of $X$.
\end{lemma}

Let $E_\lambda := E_{A, \lambda}$ and let $E:=\{E_\lambda \mid \lambda \text{ is an eigenvalue of } A\}$.

\begin{lemma} \label{primitiveroot}
Both $\Lambda$ and $\Lambda_d$ must be ${p^r}$th roots of unity, and at least one must be a primitive root.
\end{lemma}

\begin{proof}
Let $E_1$ be the eigenspace of $A$ with eigenvalue $1.$ By Lemma \ref{L3} and the group relations of $\Hz$ we have that $ B \cdot E_1 = E_\Lambda$ and $B_d  \cdot E_1 = E_{\Lambda_d}$ with the other generators acting trivially on $E_1.$ Since by Lemma \ref{eigenspacesize} $A$ has $p^r$ distinct eigenspaces, we have that $\mid \langle B \rangle \cdot E_1 \mid = p^{k}$ and $\mid \langle B_d \rangle \cdot E_1 \mid = p^{k_d}$ for some $k,k_d \leq r.$ Thus $\Lambda$ is a primitive $p^k$th root of unity and $\Lambda_d$ is a primitive $p^{k_d}$th root of unity. Finally, since there are exactly $p^r$ distinct eigenspaces of $A$,
at least one of $k,k_d$ must equal $r$.



\end{proof}

The proof of Lemma \ref{primitiveroot} shows that, in fact, at least one of $B$ and $B_d$ permutes the eigenspaces of $A$ transitively, depending on whether $\Lambda$ or $\Lambda_d$ are $p^r$th primitive roots of unity. To deal with both cases, we break the argument into two separate parts.   
\medskip

\noindent \textbf{Case 1}: $\Lambda$ is a primitive $p^r$th root of unity.\\

Since $\Lambda$ is primitive we can deduce that $E = \{ E_1, E_\Lambda, E_{\Lambda^2}, \ldots, E_{\Lambda^{p^r-1}}\}$. We know that $B$ permutes $E$ transitively and since $\Lambda$ is primitive, $B$ must correspond to a  $p^r$-cycle permutation, say $\sigma_B$. This, together with the fact that $B$ and $B_d$ commute and  that in $S_n$ only powers of $n$-cycles commute with $n$-cycles, allows us to deduce that the permutation of $E$ corresponding to $B_d$, say $\sigma_{B_d}$, is $\sigma_B^l$ for some $l$. Therefore we have that
\begin{equation} \label{lambdal}
\Lambda_d=\Lambda^l
\end{equation}
since $B_d$ sends $E_\lambda$ to $E_{\Lambda^l \lambda}$ and $[A,B_d] = \Lambda_d$. Define $D_l$ as $d$ if $d \equiv 2,3 \ (\modd 4)$ and $\frac{d-1}{4} +l$ if $d \equiv 1 \ (\modd 4)$. Then we can say that $[A_d,B_d]= \Lambda^{D_l}$.

\begin{lemma}\label{l3.5}
The eigenvalues of $B$ and $B_d$ are $p^r$th roots of unity. Moreover, $B$ has all $p^r$th roots of unity as eigenvalues.
\end{lemma}

\begin{proof}
Since the commutator relation is antisymmetric we have that $[B,A] = \Lambda^{-1}$. Therefore $B$ is $\Lambda^{-1}$-arrangeable under $A$. We recall that $1$ is an eigenvalue of $B$. Since $\Lambda$ is a primitive $p^r$th root of unity, we have, by  Lemma \ref{L3}, that $B$ has at least $p^r$ distinct eigenspaces and all $p^r$ roots of unity are eigenvalues of $B$. However, since $B$ is  $\Lambda_d^{-1}$-arrangeable under $A_d$ and $\Lambda_d$ is a $p^r$th root of unity then, by a similar argument to that of Lemma \ref{primitiveroot}, the $\langle A,A_d \rangle$ orbit of the eigenspace of $B$ with eigenvalue $1$ will be a stable subspace of $\mathbb{C}^{p^n}$. But since $\rho$ is irreducible, we have that this must be the entire space and so this action is transitive. So $B$ has at most $p^r$ distinct eigenspaces. The argument for $B_d$, using the appropriate commutator relations and without the condition that it has all $p^r$th roots of unity as eigenvectors, is similar.\\
\end{proof}

\subsection{Picking a Basis}

We now choose a basis $\TT_1$ for $E_1$ and construct the rest of the basis of $\mathbb{C}^{p^n}$  in the following way. Define  $\TT_{\Lambda^k}:= B^k \cdot \TT_1$ for $1 \leq k \leq p^r-1$  such that $\TT_{\Lambda^k}$ is a basis for $E_{\Lambda^k}.$  Then we have a basis $\TT:=\TT_1 \cup \ldots \cup \TT_{\Lambda^{p^r-1}}$ for $\mathbb{C}^{p^n}$ with respect to which $A$ and $A_d$ are diagonal. Since $\rho$ is good, $1$ is an eigenvalue of $A$ and $A_d$.  We can choose $\TT_1$ such that $(A_{d})_{1,1}=1$,

$$A = \left( \begin{array} {cccc} I_{p^m} &  &  &  \\  & \Lambda I_{p^m} &  &  \\  &  & \ddots &  \\  &  &  & \Lambda^{{p^r}-1}I_{p^m} \end{array} \right),$$

\noindent and

$$ B = \left( \begin{array} {cccc}
0_{p^m} &  &  &  P \\
 I_{p^m} & \ddots &  &  \\
    & \ddots & \ddots &  \\
       &  & I_{p^m} & 0_{p^m} \end{array} \right)$$

\noindent where $I_{p^m}$ and $0_{p^m}$ are, respectively, the identity and null matrices of size $p^m$, and for some matrix $P$ of size $p^m$. However, by Lemma \ref{l3.5}, $B^{p^r}=I$. This implies that $P=I_{p^m}$.

Since $A$ is arrangeable under $B_d$, Lemmas \ref{L3} and \ref{l3.5} imply that $B_d$ must be a generalized block permutation matrix with blocks of size $p^m$; that is it has
exactly one non-zero block in each block row and block column. Since  $[B,B_d] = I$ and $\sigma_{B_d}= \sigma_B^l$  a simple computation shows that $B_d$ is the block matrix\\

  \begin{equation}\label{Bd}
  B_d = \left( \begin{array} {cccccc} 0_{p^m} &  &  &R & & \\  & \ddots &  & & \ddots & \\  &  & \ddots & &  & R\\ R  &  &  & \ddots & &  \\ & \ddots & & & \ddots & \\ & & R & & & 0_{p^m} \end{array} \right)
  \end{equation}\\


\noindent for some matrix $R$ of size $p^m$ with respect to $\TT$, and the $R$ in the first block
column is in the $l$th block row.

Since $A_d$ is diagonal let

$$A_d = \left( \begin{array} {cccc} J_1 &  &  &  \\  & J_2 &  &  \\  &  & \ddots &  \\  &  &  & J_{p^r} \end{array} \right)$$

\noindent for some diagonal matrices $J_i, 1 \leq i \leq p^r$ of size $p^m$. Since $[A_d,B]= \Lambda^l$, $A_d$ is $\Lambda^l$-arrangeable under $B$. Therefore by Lemma \ref{L3} and for some matrix $J$ of size $p^m$, we have that

$$A_d = \left( \begin{array} {cccc} J &  &  &  \\  & \Lambda^lJ &  &  \\  &  & \ddots &  \\  &  &  & \Lambda^{(p^r-1)l}J \end{array} \right).$$

\noindent A straightforward calculation using the relation $[A_d,B_d]=\Lambda^{D_l}$ gives us the equation $JR=\Lambda^{l^2-{D_l}}RJ$. Therefore $J$ is $(\Lambda^{l^2-{D_l}})$-arrangeable under $R$.  Let $H$ be the Heisenberg group over the rational integers as defined in Equation \ref{Heisgroup} and note that $\langle J,R \rangle \cong H$. The homomorphism $x \mapsto J, y \mapsto R$ defines a representation of $H.$

Assume that $\mathbf{G}$ is a non-zero $(J,R)$-stable irreducible subspace of $E_1$, say, by Lemma \ref{simplified}, of dimension $p^k$ where $k < m$. It is clear that $\langle A \rangle \cdot S = S$ and since $\mathbf{G}$ is $J$-stable we have that $\langle A_d \rangle \cdot S = S$. It is also clear by Lemma \ref{L3} that $\langle B \rangle \cdot S = \{S, BS, B^2S, \ldots, B^{p^r-1}S\}$ where the $B^iS$ are distinct. Since $\mathbf{G}$ is $R$-stable and $\sigma_{B_d}=\sigma_B^l$ we have that $B_d \cdot S = B^l \cdot S$ and therefore $\langle B_d \rangle \cdot S \subseteq \langle B \rangle \cdot S$. We also note that $\langle A \rangle \cdot B^iS = \Lambda^iS=S$ and $\langle A_d \rangle \cdot B^iS = \Lambda^{li}S=S$ for $0\leq i \leq p^{r}-1.$ Therefore  the size of the orbit $\Hz  \cdot S$ is $p^r$ and the dimension of the subspace spanned by $ \Hz  \cdot S$ is $p^{r +k}$. Since $p^{r+k} <p^{r + m} = p^n$  we have that this is a proper stable subspace of $\rho$. But since $\rho$ is irreducible this is a contradiction.  This implies that the representation of the subgroup generated by $\langle J,R \rangle$ on $E_1$ has no stable subspaces and is therefore irreducible. Since this is a representation of $H$, by Nunley and Magid's result on the irreducible representations of $H$ \cite[Theorem 5]{NM}, we could have chosen $\TT_1$ at the start of this subsection so that

\begin{equation}\label{R} R=\left( \begin{array} {cccc} 0 &  &  & 1  \\ 1 & \ddots & &  \\  & \ddots
 & \ddots &  \\  &  & 1
 & 0 \end{array} \right) \end{equation}

\noindent and therefore

$$J=\left( \begin{array} {cccc} 1 &  &  &   \\  & \Lambda^{l^2-{D_l}} & &  \\  &
 & \ddots &  \\  &  &
 & \Lambda^{(p^{m}-1)(l^2-{D_l})} \end{array} \right).$$

\noindent By \cite[Theorem 5]{NM} we can now deduce that
\begin{equation}\label{p^mth}
\Lambda^{l^2-{D_l}} \text{is a primitive $p^m$th root of unity.}
\end{equation}

\begin{lemma} \label{L5}
$r \geq m$.
\end{lemma}

\begin{proof}
By Equation \ref{Bd} we have that

$$B_d^{p^r} = \left( \begin{array} {cccc} R^{p^r} &  &  &  \\   & R^{p^r} &  &  \\  &  & \ddots &  \\  &  &  & R^{p^r} \end{array} \right).$$
But by Lemma \ref{l3.5},
$$B_d^{p^r} = \left( \begin{array} {cccc} I_{p^m} &  &  &  \\   & I_{p^m} &  &  \\  &  & \ddots &  \\  &  &  & I_{p^m} \end{array} \right).$$
Therefore $R^{p^r} = I_{p^m}$ and, since $R$ corresponds to a $p^m$-cycle, $ r \geq m$.

\end{proof}

From Lemma \ref{L5} and  Equation \ref{p^mth} we deduce that
\begin{equation}\label{r-m}
l^2 \equiv {D_l} \ (\modd p^{r-m})
\end{equation}

\noindent and

\begin{equation}\label{r-c}
l^2 \not\equiv {D_l} \ (\mbox{mod} \  p^{r-m+1}) 
\end{equation}

\noindent for $0 \leq l < p^r$. Since if $m=0$ then any $p^m$th root of 1 is primitive, Condition \ref{r-c} applies only if $m \neq 0$. We remark that these two conditions can be expressed as
\begin{equation}\label{padic}
v_p(l^2-D_l) = r-m
\end{equation}
where $v_p(x)$ is the $p$-adic valuation of $x$.

We have now completely determined all of our matrices. That is,

$$A = \left( \begin{array} {cccc} I_{p^m} &  &  &  \\  & \Lambda I_{p^m} &  &  \\  &  & \ddots &  \\  &  &  & \Lambda^{{p^r}-1}I_{p^m} \end{array} \right),$$

$$ B = \left( \begin{array} {cccc}
0_{p^m} &  &  &  I_{p^m} \\
 I_{p^m} & \ddots &  &  \\
    & \ddots & \ddots &  \\
       &  & I_{p^m} & 0_{p^m} \end{array} \right),$$

$$A_d = \left( \begin{array} {cccc} J &  &  &  \\  & \Lambda^lJ &  &  \\  &  & \ddots &  \\  &  &  & \Lambda^{(p^r-1)l}J \end{array} \right),$$

\noindent and

$$B_d = \left( \begin{array} {cccccc} 0 &  &  &R & & \\  & \ddots &  & & \ddots & \\  &  & \ddots & &  & R\\ R  &  &  & \ddots & &  \\ & \ddots & & & \ddots & \\ & & R & & & 0 \end{array} \right)$$

\noindent where

$$ J = \left( \begin{array} {cccc} 1 &  &  &  \\  & \Lambda^{l^2-{D_l}}  &  &  \\  &  & \ddots &  \\  &  &  & \Lambda^{({p^r}-1)(l^2-{D_l})} \end{array} \right)$$

\noindent and

$$R=\left( \begin{array} {cccc} 0 &  &  & 1  \\ 1 & \ddots & &  \\  & \ddots
 & \ddots &  \\  &  & 1
 & 0 \end{array} \right).$$

\noindent There are $(1-p^{-1})p^r$ choices for the primitive root of unity $\Lambda.$ The number of choices for
$\Lambda_d=\Lambda^l$---that is, the number of the solutions to Congruences \ref{r-m} and \ref{r-c}---will be analyzed in Section \ref{pqqw}.\\

\noindent \textbf{Case 2}: $\Lambda_d$ is a primitive $p^r$th root of unity. The analysis in this case is similar to Case 1. However, in order to avoid overcounting, we also assume the condition

\begin{equation}
\Lambda \  \text{is not a primitive} \ p^r \text{th root of unity.}
\label{eq3}
\end{equation}
\\

Following the methods in Case 1, but switching the roles of $B$ and $B_d$, we obtain the matrices

$$A = \left( \begin{array} {cccc} I_{p^m} &  &  &  \\  & \Lambda_d I_{p^m} &  &  \\  &  & \ddots &  \\  &  &  & \Lambda_d^{{p^r}-1}I_{p^m} \end{array} \right),$$

$$ B_d = \left( \begin{array} {cccc}
0_{p^m} &  &  &  I_{p^m} \\
 I_{p^m} & \ddots &  &  \\
    & \ddots & \ddots &  \\
       &  & I_{p^m} & 0_{p^m} \end{array} \right),$$

 $$A_d = \left( \begin{array} {cccc} J &  &  &  \\  & \Lambda_d^{l{D_l}}J &  &  \\  &  & \ddots &  \\  &  &  & \Lambda_d^{(p^r-1)l{D_l}}J \end{array} \right),$$ and

$$B = \left( \begin{array} {cccccc} 0 &  &  &R & & \\  & \ddots &  & & \ddots & \\  &  & \ddots & &  & R\\ R  &  &  & \ddots & &  \\ & \ddots & & & \ddots & \\ & & R & & & 0 \end{array} \right)$$

where

$$ J = \left( \begin{array} {cccc} 1 &  &  &  \\  & \Lambda_d^{l^2{D_l}-1}  &  &  \\  &  & \ddots &  \\  &  &  & \Lambda_d^{({p^r}-1)(l^2{D_l}-1)} \end{array} \right)$$

\noindent and

$$R=\left( \begin{array} {cccc} 0 &  &  & 1  \\ 1 & \ddots & &  \\  & \ddots
 & \ddots &  \\  &  & 1
 & 0 \end{array} \right).$$

This allows us to recover the conditions

\begin{equation}\label{r-m2}
l^2{D_l} \equiv 1 \ (\mbox{mod} \ p^{r-m})
\end{equation}

\noindent and

\begin{equation}\label{r-c2}
l^2{D_l} \not\equiv 1 \ (\mbox{mod} \ p^{r-m+1}) 
\end{equation}
\noindent for $0 \leq l <p^r$,
where Condition \ref{r-c2} does not apply when $m=0$.  We remark that these two equations can be written as
\begin{equation}\label{padic2}
v_p(l^2{D_l}-1)=r-m.
\end{equation}
This concludes the case distinctions.

Call Equations \ref{r-m} and \ref{r-c}, the \emph{Case-1-Conditions} and Equations \ref{r-m2} and \ref{r-c2} the \emph{Case-2-Conditions}. Note that the Case-2-Conditions imply that $l$ is invertible modulo $p$. However, since we are assuming Condition \ref{eq3} in Case 2, we have that $p \mid l$.  Thus, there are only solutions to the Case-2-Conditions when $r=m$.\\

Now we check that all matrices of this form give us an irreducible representation of $\Hz$. This is clear; take matrices $\{A,B,A_d,B_d,\Lambda,\Lambda_d\}$ of the forms above. Then an easy calculation shows that the associated relations for $\Hz$ hold. Since $z, z_d$ are commutators, they remain fixed under twisting. Since the matrices $A,B,A_d,$ and $B_d$ are determined by $\Lambda$ and $\Lambda_d$, two such ordered sets of matrices define twist-equivalent representations if and only if they coincide.

The form of $\zeta_{\Hz ,p}(s)$ depends on $p$ and $d$. This subsection therefore implies the following:
\begin{proposition}\label{correspondence}
There is a 1-1 correspondence between the set of $p^n$-dimensional twist-isoclasses of $\Hz$ and the set of solutions to Conditions \ref{r-m}, \ref{r-c}  in Case 1 or  \ref{eq3}, \ref{r-m2}, \ref{r-c2} in Case 2 such that  $r+m=n, 0\leq m \leq r \leq n.$
\end{proposition}

%
%

\subsection{Calculating the Zeta Function}\label{pqqw}

Proposition \ref{correspondence} allows us to calculate each $\zeta_{\Hz,p}(s)$ by counting solutions to Conditions \ref{r-m} and \ref{r-c} in Case 1 and \ref{r-m2} and \ref{r-c2} in Case 2  modulo $p^n$ for each $n \geq 0.$ To count the number of solutions we use Hensel's Lemma to lift solutions of the Conditions modulo $p$ if $p$ is not ramified; if $p$ is ramified, the computation is nevertheless straightforward. We demonstrate the computations and then summarize the results in a table. We note three things: there are always $(1-p^{-1})p^r$ choices for $\Lambda$ in Case 1 and $\Lambda_d$ in Case 2, in Case 2 it is easy to see there are $(1-p^{-1})p^{n-1}$ solutions when $r=m$ and $0$ otherwise, and there is only $1$ irreducible twist-isoclass when $n=0$. The following cases assume $n \neq 0$.\\

For ease of understanding we tabulate the Conditions with their reference numbers:\\

\begin{tabular}{|l|c|c|}
\hline
Case 1 & $\Lambda \text{ is a primitive } p^r\text{th root of unity}$&\\
    &$l^2 \equiv {D_l} \ (\modd p^{r-m})$ & (\ref{r-m})\\

    & $l^2 \not\equiv {D_l} \ (\mbox{mod} \  p^{r-m+1}) \ \text{for} \  r \neq n$ & (\ref{r-c})  \\
\hline
Case 2 & $\Lambda_d \  \text{is a primitive} \ p^r \text{th root of unity}$ & \\
    &$\Lambda \  \text{is not a primitive} \ p^r \text{th root of unity}$ & \\
 &$l^2{D_l} \equiv 1 \ (\mbox{mod} \ p^{r-m})$ & (\ref{r-m2})\\
 &$l^2{D_l} \not\equiv 1 \ (\mbox{mod} \ p^{r-m+1}) \ \text{for} \  r \neq n $ & (\ref{r-c2})\\
\hline
\end{tabular}\\
\medskip

Assume $p$ is inert. This implies that $d$ is not a square modulo $p.$ Thus, in this case, there are no solutions to the Case-1 or Case-2-Conditions unless $r=m$. Given $\Lambda$, and since $r+m=n$, there are $p^{\frac{n}{2}}$ choices for $\Lambda_d$ in Case 1. Therefore, with Case 2 contributing $(1-p^{-1})p^{\frac{n}{2}}(p^{\frac{n}{2}-1})$ in the even case,

$$r_{p^n} =
\begin{cases}
(1-p^{-1})p^{\frac{n}{2}}p^{\frac{n}{2}}+(1-p^{-1})p^{\frac{n}{2}}(p^{\frac{n}{2}-1})&\text{\quad for even } n\\
0 &\text{\quad for odd } n
\end{cases}
$$

\noindent and

\begin{align*}
 \zeta_{\Hz,p}(s) &= \sum_{n=0}^{\infty} r_{p^n}p^{-ns} = 1 + \sum_{m=1}^{\infty}(1-p^{-2})(p^{2-2s})^m = \frac{1-p^{-2s}}{1-p^{2-2s}}.
\end{align*}


Assume $p$ splits. There are two solutions to the equation $l^2 \equiv {D_l} \ (\modd p)$ and Hensel's Lemma allows us to ``lift" these solutions to solutions in $\mathbb{Z}/p^{r-m}\mathbb{Z}$, thus giving us the two unique solutions to $l^2 \equiv {D_l} \ (\modd p^{r-m})$. When $r=n$, there are two solutions to Condition \ref{r-m} and therefore $2(1-p^{-1}){p^n}$ choices for the pair $\Lambda$ and $\Lambda_d$ in the Case-1-Conditions. If, for fixed $r$ and $m$,  $r>m$ and $m >0$ then there are two solutions in $\mathbb{Z}/p^{r-m}\mathbb{Z}$ to Condition \ref{r-m} and therefore $2p^m$ solutions for $0\leq l \leq p^r -1$. Of these solutions, all but $2p^{m-1}$ satisfy Condition \ref{r-c}. Therefore, given $\Lambda$, there are $2(1-p^{-1})p^m$ choices for $\Lambda_d$ and $2(1-p^{-1})^2p^n$ choices for the pair $\Lambda$ and $\Lambda_d$ in Case 1. If $r=m=\frac{n}{2}$ then there are $p^m$ solutions to Condition \ref{r-m}, of which  all but $2p^{m-1}$ satisfy Condition \ref{r-c}. Therefore there are $(1-2p^{-1})(1-p^{-1}){p^n}$ choices for the pair $\Lambda$ and $\Lambda_d$ in Case 1. Summing all cases together, and noting that the Case 2 contribution is 0 in the odd case and $(1-p^{-1})p^{n-1}$ in the even case, we have

\begin{equation}
r_{p^n}=
\begin{cases}

2(1-p^{-1})p^n + \frac{n-1}{2}2(1-p^{-1})^2p^n \text{\qquad if $n$ is odd}\\
2(1-p^{-1})p^n + \frac{n-2}{2}2(1-p^{-1})^2p^n \\
\ + (1-2p^{-1})(1-p^{-1}){p^n} + (1-p^{-1}){p^{n-1}} \text{\quad if $n$ is even}.

\end{cases}
\end{equation}

\noindent Strikingly, in both cases this simplifies to
$$r_{p^n}= ((1+p^{-1})+(1-p^{-1})n)(1-p^{-1})p^n$$

\noindent and therefore

\begin{align*}
 \zeta_{\Hz,p}(s) &= \sum_{n=0}^{\infty} r_{p^n}p^{-ns}\\
  &= 1+ \sum_{n=1}^{\infty}(1-p^{-1})(p^{1-s})^n [(1+p^{-1})+(1-p^{-1})n] \\
  &= 1+ \sum_{n=1}^{\infty}(1-p^{-2}) (p^{1-s})^n + \sum_{n=1}^{\infty}(1-p^{-1})^{2} n(p^{1-s})^n  \\
  &= 1+ \dfrac{(1-p^{-2})p^{1-s}}{1-p^{1-s}}+\dfrac{(1-p^{-1})^{2}p^{1-s}}{(1-p^{1-s})^2} \\
   &= \left(\dfrac{1-p^{-s}}{1-p^{1-s}}\right)^2.
\end{align*}

Assume $p$ is ramified. This is the case if  $d \equiv 0 \ (\modd p)$ for any $d$ or if $p=2$ and $d \equiv 2,3 \ (\modd 4)$. Then there are solutions to the Case-1 and Case-2 Conditions only when $r-m =0$ or $r-m=1$. If $d \equiv 2,3 \ (\modd 4)$ and $d \equiv 0 \ (\modd p)$ then $0$ is the solution to $l^2 \equiv d \ (\modd p)$ but since $d$ is squarefree $d \equiv kp \mod p^2$ for some invertible $k \in \mathbb{Z}/p^2\mathbb{Z}.$ Therefore $l^2=d$ has no solutions modulo $p^2$. If $d \equiv 2,3 \ (\modd 4)$ and $p=2$ then $l^2 \equiv d$ has a solution modulo $p$ but since $d$ is a quadratic non-residue modulo $4$ it has no solutions modulo $4$. If $d \equiv 1 \ (\modd 4)$ and $d \equiv 0 \ (\modd p)$ then $l^2-l+\frac{d-1}{4}=0$ has the unique solution $l \equiv 2^{-1} \ (\modd p)$. And since $d$ is squarefree $d \equiv kp \ (\modd p^2)$ for some invertible $k \in \mathbb{Z}/p^2\mathbb{Z}$. Then the above equation can be rearranged to the form $(2l-1)^2 \equiv kp \ (\modd p^2)$, which clearly has no solutions.  Then if $r-m = 0$, then there are $p^m$ solutions to Condition \ref{r-m} and all but $p^{m-1}$ of these satisfy Condition \ref{r-c}. Therefore, given $\Lambda$ there are $(1-p^{-1})p^m$ choices for $\Lambda_d$ and $(1-p^{-1}) (1-p^{-1})p^n$ choices for the pair $\Lambda$ and $\Lambda_d$ in Case 1 and, as usual, $(1-p^{-1})p^{n-1}$ choices in Case 2. If $r-m = 1$, a similar calculation to the ones above yields that there are $(1-p^{-1})p^n$ choices in Case 1. Therefore

$$r_{p^n} = (1-p^{-1})p^n$$

\noindent and thus

\begin{align*}
 \zeta_{\Hz,p}(s) &{}= 1+\sum_{n=1}^{\infty}(1-p^{-1})p^np^{-ns}\\
                  &{}=  1+(1-p^{-1})\sum_{n=1}^{\infty}(p^{1-s})^n\\
                  &{}= \frac{1-p^{-s}}{1-p^{1-s}}.
\end{align*}

The preceding results are tabulated for easy reference. Note that, for $k \geq 1,$ we have that  $\phi(p^k)=(1-p^{-1})p^k$:

\begin{tabular}{|l||c|c|}
\hline
prime behaviour & $r_{p^n}; n > 0$ &  $\zeta_{\Hz,p} (s)$ \\
\hline
& &\\[-1em]

\hline
 & & \\[-1.8ex]
inert & $\begin{cases} (1+p^{-1}) \phi({p^{n}}) &\text{ for even } n\\ 0 &\text{ for odd } n \end{cases}$  &  $\dfrac{1-p^{-2s}}{1-p^{2-2s}}$ \\
& & \\[-1.8ex]
\hline
& & \\[-1.8ex]
splits & $\phi(p^n)\big(n(1-p^{-1}) + 1 + p^{-1}\big)$ & $ \left(\dfrac{1-p^{-s}}{1-p^{1-s}}\right)^2$ \\
& & \\[-1.8ex]
\hline
& & \\[-1.8ex]
ramified & $\phi(p^n)$ & $\dfrac{1-p^{-s}}{1-p^{1-s}}$\\

\hline
\end{tabular}

%
%
%
%
%
\vspace{1cm}
By Proposition \ref{oror} we can say that

$$ \zeta_{\Hz} (s) = \dfrac {\zeta^{\mathbf{D}}_{\mathbb{Q}(\dd)} (s-1)}{\zeta^{\mathbf{D}}_{\mathbb{Q}(\dd)} (s)}.$$

\medskip
{\em Acknowledgements}: The authour would like to thank his supervisor Ben Martin for overall guidance and the many great suggestions for the structure of the paper. The author would also like to show appreciation to the reviewer for their suggestions.

\medskip

\renewcommand{\bibname}{References}
\bibliographystyle{plain}
\bibliography{prop2good}{}

\end{document}